\documentclass[a4paper,12pt,showkeys]{amsart}
\usepackage[english]{babel}
\usepackage{amsmath,amssymb,amsthm}
\usepackage[dvipdfmx]{graphicx}
\usepackage{color}
\usepackage{lineno}

\newtheorem{theorem}{Theorem}[section]
\newtheorem{prop}[theorem]{Proposition}
\newtheorem{lemma}[theorem]{Lemma}
\newtheorem{cor}[theorem]{Corollary}
\newtheorem{example}[theorem]{Example}
\theoremstyle{definition}
\newtheorem{defn}[theorem]{Definition}
\newtheorem{rem}[theorem]{Remark}

\makeatletter

\@addtoreset{equation}{section}
\makeatother
\begin{document}
\title{Dwork hypersurfaces of degree six and Greene's hypergeometric function}
\author{Satoshi Kumabe}
\date{}
%\date{2020年2月17日}
%タイトルなどの出力
\address{Mathematical Institute, Kyushu University, Motooka, Fukuoka 819-0395, Japan }
\email{kuma511ssk@gmail.com}
\maketitle
%アブストラクト
\begin{abstract}
In this paper, we give a formula for the number of rational points on the Dwork hypersurfaces of degree six over finite fields by using Greene's finite-field hypergeometric function, which is a generalization of Goodson's formula for the Dwork hypersurfaces of degree four.
Our formula is also a higher-dimensional and a finite field analogue of Matsumoto-Terasoma-Yamazaki's formula. Furthermore, we also explain the relation between our formula and Miyatani's formula. 
\end{abstract}
\footnote{
}

\section{Introduction}

It is an interesting problem to express the number of rational points on certain varieties over finite fields by using finite-field hypergeometric functions. Finite-field hypergeometric functions were introduced independently by Greene \cite{Greene}, Katz \cite{katz}, Koblitz \cite{Koblitz} and McCarthy \cite{transformation}. For example, in \cite{McCarthy F_p point}, McCarthy gave a formula for the Dwork hypersurfaces over finite fields by using his hypergeometric functions. In \cite{Salerno}, Salerno gave a formula for diagonal hypersurfaces, which are generalizations of the Dwork hypersurfaces, by using Katz's hypergeometric functions. 

In {\cite[Theorem 1.1]{case4}}, Goodson gave a formula for the number of rational points on the Dwork hypersurfaces of degree four over finite fields by using Greene's hypergeometric functions and Jacobi sums. Furthermore, in {\cite[Theorem 1.2]{cased}}, she also gave a similar formula in case of odd degree by Greene's hypergeometric functions and Gauss sums. The purpose of this paper is to extend Goodson's result to the Dwork hypersurfaces of degree six. We give the formula by using Greene's hypergeometric functions and Jacobi sums. {\cite[Theorem 1.1]{case4}} and our formula are higher-dimensional and finite field analogues of the formula of Matsumoto-Terasoma-Yamazaki \cite[Theorem 1]{MTY}, for the complex periods of a Hesse cubic curve, that is, the Dwork hypersurfaces of degree three. $($For more details, see Remark \ref{analogy}.$)$ 

Now, we explain our formula precisely. First, we recall Goodson's results. Let $\mathbb{F}_q$ be the finite field with $q=p^e$ elements, where $p$ is a prime number. 
Let $d$ be a positive integer. For $\lambda \in \mathbb{F}_q$, we define the Dwork hypersurface $X_{\lambda}^d$ by the projective equation
\begin{equation*}
x_1^d+x_2^d+\cdots +x_d^d=d\lambda x_1x_2\cdots x_d
\end{equation*}
over $\mathbb{F}_q$. 
Let $\widehat{\mathbb{F}}_q^\times$ be the group of characters on $\mathbb{F}_q^\times$ in $\mathbb{C}^{\times}$ and $\omega$ a generator of $\widehat{\mathbb{F}}_q^{\times}$ which we fix throughout this paper. For a character $\chi \in \widehat{\mathbb{F}}^{\times}_q$, we extend it by putting $\chi(0)=0$. We define the trivial character $\epsilon \in \widehat{\mathbb{F}}_q^\times$ by putting $\epsilon(x)=1$ for any $x\in \mathbb{F}_q^\times$ and extend it by putting $\epsilon(0)=0$.
Then for $\omega$ $\in \widehat{\mathbb{F}}_q^\times$, we define the Gauss sum $g(\omega)$ by

\begin{align*}
g(\omega):=\sum_{x\in \mathbb{F}_q}\omega(x)\mathrm{exp}\left(\frac{2\pi \sqrt{-1}\cdot \mathrm{tr}(x)}{p}\right),
\end{align*}
\normalsize
where $\mathrm{tr}$ is the trace map from $\mathbb{F}_q$ to $\mathbb{F}_p$.
Note that we obtain $g(\epsilon)=-1$ from $\epsilon(0)=0$.
Furthermore, for characters $\chi, \psi \in \widehat{\mathbb{F}}_q^{\times}$, we define the Jacobi sum by

\begin{align*}
J(\chi, \psi):=\sum_{x\in \mathbb{F}_q}\chi(x)\psi(1-x)=\sum_{x+y=1}\chi(x)\psi(y).
\end{align*}
\normalsize
More generally, for characters $\chi_1, \chi_2, \ldots , \chi_n\in \widehat{\mathbb{F}}_q^\times$, we define the Jacobi sum by

\begin{align*}
J(\chi_1, \chi_2, \ldots , \chi_n)
=\sum_{\substack{x_1, \ldots x_n\in \mathbb{F}_q\\x_1+\cdots +x_n=1}}\chi_1(x_1)\cdots \chi_n(x_n).
\end{align*}
\normalsize
Next, we define Greene's hypergeometric function. For $A,B \in \widehat{\mathbb{F}}_q^{\times}$, we define the normalized Jacobi sum by 

\begin{align*}
\binom{A}{B}:= \frac{B(-1)}{q}\sum_{x\in \mathbb{F}_q}A(x)\overline{B}(1-x) =\frac{B(-1)}{q}J(A,\overline{B}),
\end{align*}
\normalsize
where $\overline{B}$ is the complex conjugate of $B$. 
Then for $n\geq 1$, $A_0,A_1, \ldots ,A_n, \\
B_1,B_2,\ldots ,B_n \in \widehat{\mathbb{F}}_q^{\times}$ and $x \in \mathbb{F}_q$, we define Greene's hypergeometric function ${}_{n+1}F_n$ by
\scriptsize
\[
{}_{n+1}F_n \left(
\begin{array}{cccc|c}
A_0, & A_1, & \ldots & A_n \\
& B_1, & \ldots & B_n
\end{array}
\ x \right)
_q := \begin{cases}
\displaystyle
\frac{q}{q-1}\sum_{\chi\in \widehat{\mathbb{F}}_q^\times} \binom{A_0\chi}{\chi} \binom{A_1\chi}{B_1\chi} \cdots \binom{A_n\chi}{B_n\chi} \chi (x) & (n\geq 2) \\
\\
\vspace{1mm} \\
\displaystyle
\epsilon (x)\frac{A_1B_1(-1)}{q}\sum_{y\in \mathbb{F}_q}A_1(y)\overline{A_1}B_1(1-y)\overline{A_0}(1-xy) & (n=1).
\end{cases}
\]
\normalsize
Then, Goodson obtained the following results.
\begin{theorem}[{\cite[Theorem 1.1]{case4}}]\label{case4}
Let $q=p^e$ be a power of a prime number such that $q$ is congruent to $1$ modulo $4$ and $\omega$ a generator of $\widehat{\mathbb{F}}^\times_q$. We put $t=(q-1)/4$.
For $\lambda\in \mathbb{F}_q$ with $\lambda \neq 0$ and $\lambda^4\neq 1$, we have

\begin{equation*}
\begin{split}
\#X_{\lambda}^4(\mathbb{F}_q)&=\frac{q^3-1}{q-1}+12q\omega^t(-1)\omega^{2t}(1-\lambda^4) \\
&\vspace{5mm} \\
&+q^2 \cdot{}_3F_2
\left(
\begin{array}{ccc|c}
\omega^t & \omega^{2t} & \omega^{3t} \\
& \epsilon & \epsilon
\end{array}
\ \frac{1}{\lambda^4} \right)
_q+3q^2\binom{\omega^{3t}}{\omega^t} \ {}_2F_1
\left(
\begin{array}{cc|c}
\omega^{3t} & \omega^{t} \\
& \omega^{2t}
\end{array}
\ \frac{1}{\lambda^4} \right)
_q.
\end{split}
\end{equation*}
\normalsize
\end{theorem}
\begin{theorem}[{\cite[Theorem 1.4]{cased}}]\label{case5}
Let $q=p^e$ be a power of a prime number such that $q$ is congruent to $1$ modulo $5$ and $\omega$ a generator of $\widehat{\mathbb{F}}_q^{\times}$. We put $t=(q-1)/5$. For $\lambda\in \mathbb{F}_q$ with $\lambda \neq 0$ and $\lambda^5 \neq 1$, we have
\begin{equation*}
\begin{split}
\#X_{\lambda}^5(\mathbb{F}_q)&=\frac{q^4-1}{q-1} +q^3 \cdot{}_4F_3
\left(
\begin{array}{cccc|c}
\omega^t & \omega^{2t} & \omega^{3t} & \omega^{4t} \\
& \epsilon & \epsilon & \epsilon
\end{array}
\ \frac{1}{\lambda^5} \right)
_q \\
& \vspace{5mm} \\ 
&+20q^2\cdot{}_2F_1
\left(
\begin{array}{cc|c}
\omega^{2t} & \omega^{3t} \\
& \epsilon
\end{array}
\ \frac{1}{\lambda^5} \right)
_q +20q^2\cdot{}_2F_1
\left(
\begin{array}{cc|c}
\omega^{t} & \omega^{4t} \\
& \epsilon
\end{array}
\ \frac{1}{\lambda^5} \right)
_q \\
&\vspace{5mm} \\
&+ 30q^2\cdot{}_2F_1
\left(
\begin{array}{cc|c}
\omega^t & \omega^{3t} \\
& \omega^{4t}
\end{array}
\ \frac{1}{\lambda^5} \right)
_q +30q^2\cdot{}_2F_1
\left(
\begin{array}{cc|c}
\omega^{t} & \omega^{2t} \\
& \omega^{3t}
\end{array}
\ \frac{1}{\lambda^5} \right)
_q.
\end{split}
\end{equation*}
\end{theorem}
In {\cite[Theorem 1.2]{cased}}, she also explained the formula for the Dwork hypersurfaces of odd degree in terms of Greene's hypergeometric function. We remark that the coefficients of Greene's hypergeometric functions in her formula are products of Gauss sums. From a comparison with Matsumoto-Terasoma-Yamazaki's formula for the periods of the Hesse cubic curve over $\mathbb{C}$, the author considers that their coefficients should be written by Jacobi sums. (See also remark \ref{analogy}.) 

%since the periods of the Hesse cubic curve over $\mathbb{C}$ are expressed by the beta functions and the hypergeometric series over $\mathbb{C}$. (See also remark \ref{analogy}.)} 

In this paper, we consider the Dwork hypersurfaces of degree six. We put $t:=(q-1)/6$, $\omega^t:=\omega_6$, $\omega^{2t}:=\omega_3$, and $\omega^{3t}:=\omega_2$. 
The main result of this paper is the following.

\begin{theorem}\label{main theorem}
Let $q=p^e$ be a power of a prime number such that $q$ is congruent to $1$ modulo $6$. For $\lambda \in \mathbb{F}_q$ with $\lambda\neq 0$ and $\lambda^6 \neq 1$, we have
\footnotesize
\begin{equation*}
\begin{split}
\#&X_{\lambda}^6(\mathbb{F}_q)=\frac{q^5-1}{q-1}+360q^2\omega_2(1-\lambda^6) \\
&+q^4\cdot {}_5F_4
\left(
\begin{array}{ccccc|c}
\omega_6 & \omega_3 & \omega_2 & \overline{\omega}_3 & \overline{\omega}_6\\
& \epsilon & \epsilon & \epsilon & \epsilon \\
\end{array}
\ \frac{1}{\lambda^6} \right)_q +30q^3 \omega_6(-1)\cdot {}_3F_2
\left(
\begin{array}{ccc|c}
\omega_3 & \omega_2 & \overline{\omega}_3 \\
& \epsilon & \epsilon \\
\end{array}
\ \frac{1}{\lambda^6} \right)_q \\
&+30q^3\cdot {}_3F_2
\left(
\begin{array}{ccc|c}
\omega_6 & \omega_2 & \overline{\omega}_6 \\
& \epsilon & \epsilon \\
\end{array}
\ \frac{1}{\lambda^6} \right)_q \\
&-15q^3\omega_6(-1) J(\omega_2, \overline{\omega}_3,\overline{\omega}_6)\cdot {}_4F_3
\left(
\begin{array}{cccc|c}
\omega_6 & \overline{\omega}_6 & \overline{\omega}_3 & \omega_3 \\
& \epsilon & \epsilon & \omega_2 \\
\end{array}
\ \frac{1}{\lambda^6} \right)_q \\
&-20q^3\omega_6(-1) J(\omega_6,\omega_3,\omega_2)\cdot {}_4F_3
\left(
\begin{array}{cccc|c}
\omega_6 & \omega_2 & \overline{\omega}_3 & \overline{\omega}_6 \\
& \epsilon & \omega_3 & \omega_3 \\
\end{array}
\ \frac{1}{\lambda^6} \right)_q \\
&+60q^2 \omega_6(-1)J(\omega_6, \omega_6, \overline{\omega}_3)J(\omega_2,\overline{\omega}_3,\overline{\omega}_6)\cdot {}_3F_2
\left(
\begin{array}{ccc|c}
\omega_6 & \overline{\omega}_3 & \omega_2 \\
& \epsilon & \overline{\omega}_6 \\
\end{array}
\ \frac{1}{\lambda^6} \right)_q \\
&+60q^2J(\omega_3,\omega_3,\omega_3)J(\omega_2, \overline{\omega}_3, \overline{\omega}_6)\cdot{}_3F_2
\left(
\begin{array}{ccc|c}
\omega_3 & \overline{\omega}_6 & \omega_2 \\
& \epsilon & \omega_6 \\
\end{array}
\ \frac{1}{\lambda^6} \right)_q \\
&+90q^3\cdot{}_3F_2
\left(
\begin{array}{ccc|c}
\omega_2 & \overline{\omega}_3 & \overline{\omega}_6 \\
& \omega_6 & \omega_3 \\
\end{array}
\ \frac{1}{\lambda^6} \right)_q \\
&-30q^2J(\omega_6,\omega_6)J(\omega_6, \omega_3, \omega_2)\cdot {}_3F_2
\left(
\begin{array}{ccc|c}
\omega_6 & \omega_2 & \overline{\omega}_6 \\
& \omega_3 & \overline{\omega}_3 \\
\end{array}
\ \frac{1}{\lambda^6} \right)_q \\
&-120q^2J(\omega_6,\omega_3,\omega_2)\cdot {}_2F_1
\left(
\begin{array}{cc|c}
\omega_6 & \omega_3 \\
& \epsilon \\
\end{array}
\ \frac{1}{\lambda^6} \right)_q-120q^2 J(\omega_2,\overline{\omega}_3,\overline{\omega}_6)\cdot {}_2F_1
\left(
\begin{array}{cc|c}
\overline{\omega}_3 & \overline{\omega}_6 \\
& \epsilon \\
\end{array}
\ \frac{1}{\lambda^6} \right)_q \\
&-180q^2J(\omega_6,\omega_3, \omega_2)\cdot {}_2F_1
\left(
\begin{array}{cc|c}
\omega_3 & \overline{\omega}_3 \\
& \omega_2 \\
\end{array}
\ \frac{1}{\lambda^6} \right)_q -180q^2J(\omega_6,\omega_3, \omega_2)\cdot {}_2F_1
\left(
\begin{array}{cc|c}
\omega_3 & \overline{\omega}_6 \\
& \overline{\omega}_3 \\
\end{array}
\ \frac{1}{\lambda^6} \right)_q .
\end{split}
\end{equation*}
\normalsize 
\end{theorem}
\begin{rem}
The right hand side of this formula does not depend on the choice of $\omega$. However, each term of the right hand side of this formula may depend on the choice of $\omega$.
\end{rem}

\begin{rem}\label{analogy}
One of the novelties of the above result is an expression by using the Jacobi sum. In {\cite[Theorem 1]{MTY}}, Matsumoto-Terasoma-Yamazaki gave the formula for the periods of the Hesse cubic curve over $\mathbb{C}$ by the hypergeometric series and the beta functions. The Jacobi sum is an analogue of the beta function. Hence, Theorem \ref{main theorem} and {\cite[Theorem 1.1]{case4}} are higher-dimensional and finite field analogues of the formula due to Matsumoto-Terasoma-Yamazaki. Furthermore, Theorem \ref{main theorem} and {\cite[Theorem 1.1]{case4}} suggest that {\cite[Theorem 1.2]{cased}} can be rewritten by using the Jacobi sum.
\end{rem}

Finally, we explain the proof of our formula. In the same way as the proof of Goodson's, our proof is based on Koblitz's formula for the number of rational points on diagonal hypersurfaces. In \cite[page 145, lines 2 to 4]{cased}, Goodson pointed out a possibility to deduce her formula from Miyatani's formula in terms of McCarthy's hypergeometric functions since the relation between Greene's hypergeometric functions and McCarthy's is known in this case. (We remark that the coefficients of hypergeometric functions in his formula are not Jacobi sums but products and quotients of Gauss sums. (cf. Theorem \ref{mmiyatani} in Appendix A.)) In Appendix A, we give another proof of our formula
based on Miyatani's formula, which is simpler but we consider that the proof based on Koblitz's formula also has its own value since  it is more elementary and self-contained.

\section{example}
\begin{example} 
We define the character $\omega\in \widehat{\mathbb{F}}_{13}^\times$ by $\omega(\overline{2}^k)=\mathrm{exp}(k\pi\sqrt{-1}/6)$. $($Note that $\overline2\in \mathbb{F}_{13}$ is a generator of $\mathbb{F}_{13}^\times$.$)$ 
Then $\omega$ is a generator of $\widehat{\mathbb{F}}_{13}^\times$. We put $\zeta=\mathrm{exp}(2\pi\sqrt{-1}/12)$. For $\lambda\in \mathbb{F}_q$ with $\lambda \neq 0$ and $\lambda^6 \neq 1$, we obtain
\tiny
\begin{equation*}
\begin{split}
\#&X_{\lambda}^6(\mathbb{F}_{13})=\frac{13^5-1}{13-1}+360\cdot13^2\omega_2(1-\lambda^6)+13^4\cdot {}_5F_4
\left(
\begin{array}{ccccc|c}
\omega_6 & \omega_3 & \omega_2 & \overline{\omega}_3 & \overline{\omega}_6\\
& \epsilon & \epsilon & \epsilon & \epsilon \\
\end{array}
\ \frac{1}{\lambda^6} \right)_{13} \\
&+30\cdot13^3\cdot {}_3F_2
\left(
\begin{array}{ccc|c}
\omega_3 & \omega_2 & \overline{\omega}_3 \\
& \epsilon & \epsilon \\
\end{array}
\ \frac{1}{\lambda^6} \right)_{13}+30\cdot13^3\cdot {}_3F_2
\left(
\begin{array}{ccc|c}
\omega_6 & \omega_2 & \overline{\omega}_6 \\
& \epsilon & \epsilon \\
\end{array}
\ \frac{1}{\lambda^6} \right)_{13} \\
&-15\cdot13^3 (4\zeta^2-1)\cdot{}_4F_3
\left(
\begin{array}{cccc|c}
\omega_6 & \overline{\omega}_6 & \overline{\omega}_3 & \omega_3 \\
& \epsilon & \epsilon & \omega_2 \\
\end{array}
\ \frac{1}{\lambda^6} \right)_{13} \\
&-20\cdot13^3 (-4\zeta^2+3)\cdot{}_4F_3
\left(
\begin{array}{cccc|c}
\omega_6 & \omega_2 & \overline{\omega}_3 & \overline{\omega}_6 \\
& \epsilon & \omega_3 & \omega_3 \\
\end{array}
\ \frac{1}{\lambda^6} \right)_{13} \\
&+60\cdot13^2 (\zeta^2-4)(4\zeta^2-1)\cdot{}_3F_2
\left(
\begin{array}{ccc|c}
\omega_6 & \overline{\omega}_3 & \omega_2 \\
& \epsilon & \overline{\omega}_6 \\
\end{array}
\ \frac{1}{\lambda^6} \right)_{13} \\
&+60\cdot13^2(3\zeta^2+1)(4\zeta^2-1)\cdot{}_3F_2
\left(
\begin{array}{ccc|c}
\omega_3 & \overline{\omega}_6 & \omega_2 \\
& \epsilon & \omega_6 \\
\end{array}
\ \frac{1}{\lambda^6} \right)_{13} \\
&+90\cdot13^3\cdot {}_3F_2
\left(
\begin{array}{ccc|c}
\omega_2 & \overline{\omega}_3 & \overline{\omega}_6 \\
& \omega_6 & \omega_3 \\
\end{array}
\ \frac{1}{\lambda^6} \right)_{13}-30\cdot13^2(-\zeta^2+4)(-4\zeta^2+3)\cdot{}_3F_2
\left(
\begin{array}{ccc|c}
\omega_6 & \omega_2 & \overline{\omega}_6 \\
& \omega_3 & \overline{\omega}_3  \\
\end{array}
\ \frac{1}{\lambda^6} \right)_{13} \\
&-120\cdot13^2(-4\zeta^2+3)\cdot{}_2F_1
\left(
\begin{array}{cc|c}
\omega_6 & \omega_3 \\
& \epsilon \\
\end{array}
\ \frac{1}{\lambda^6} \right)_{13}-120\cdot13^2 (4\zeta^2-1)\cdot {}_2F_1
\left(
\begin{array}{cc|c}
\overline{\omega}_3 & \overline{\omega}_6 \\
& \epsilon \\
\end{array}
\ \frac{1}{\lambda^6} \right)_{13} \\
&-180\cdot13^2(-4\zeta^2+3)\cdot {}_2F_1
\left(
\begin{array}{cc|c}
\omega_3 & \overline{\omega}_3 \\
& \omega_2 \\
\end{array}
\ \frac{1}{\lambda^6} \right)_{13}-180\cdot13^2(-4\zeta^2+3)\cdot {}_2F_1
\left(
\begin{array}{cc|c}
\omega_3 & \overline{\omega}_6 \\
& \overline{\omega}_3 \\
\end{array}
\ \frac{1}{\lambda^6} \right)_{13}.
\end{split}
\end{equation*}
\normalsize 
\end{example}

\section{The proof by Koblitz's formula}\label{proof by Koblitz's formula}
\subsection{Identities for the Gauss sum}
In this subsection, we give identities for the Gauss sum. First, we recall the Hasse-Davenport product relation.
\begin{theorem}[{\cite[Theorem 10.1]{lang}}]\label{hassedave}
Let $m$ be a positive integer and let $q$ be a power of a prime number such that $q$ is congruent to $1$ modulo $m$. For a character $\chi \in \widehat{\mathbb{F}}_q^\times$ of order $m$
 and a character $\psi \in \widehat{\mathbb{F}}_q^\times$, we have 
\begin{equation*}
\prod_{i=0}^{m-1}g(\chi^i \psi)=-g(\psi^m)\psi^{-m}(m)\prod_{i=0}^{m-1}g(\chi^i).
\end{equation*}
\end{theorem} 
\begin{cor}\label{sekikousiki}
Let $q$ be a power of a prime number such that $q$ is congruent to $1$ modulo $6$. For j $\in \mathbb{Z}$ and $t=(q-1)/6$, we have
\begin{equation*}
g(\omega^{6j})=\frac{\prod_{i=0}^{5}g(\omega^{it+j})}{\omega^{-6j}(6)\prod_{i=1}^{5}g(\omega^{it})}.
\end{equation*}
\end{cor}
\begin{proof}
This follows from Theorem \ref{hassedave} applying for $m=6$, $\chi=\omega^t$, and $\psi=\omega^j$.
\end{proof}
We use the following lemma to prove Theorem $\ref{main theorem}$.
\begin{lemma}\label{turai}
Let $a,b$ be multiples of $t$. Then, we have
\begin{equation*}
\sum_{j=0}^{q-2}g(\omega^{j+a})g(\omega^{-j+b})\omega^j(-1)\omega^{6j}(\lambda)=(q-1)g(\omega^{a+b})\omega^b(-1)\omega^{-(a+b)}(1-\lambda^6).
\end{equation*}
\end{lemma}
\begin{proof}
We can prove this result similarly as in \cite[Proposition 3.1]{cased}. (Just replace 4 with 6 everywhere.)
\end{proof}
\subsection{Koblitz's formula for diagonal hypersurfaces}

In this subsection, we recall the general formula by Koblitz. Koblitz gave a formula for the number of $\mathbb{F}_q$-rational points on diagonal hypersurfaces 
$$D_{\lambda}: x_1^d+x_2^d+\cdots +x_n^d-d\lambda x^{h_1}_1x^{h_2}_2\cdots x^{h_n}_n=0,$$ where $d\mid q-1$, $h_1+\cdots +h_n=d$ and $\gcd (d, h_1, \ldots, h_n)=1$.  Let $W$ be the set of all $n$-tuples $\boldsymbol{w}=(w_1, \ldots, w_n)$ of the elements of $\mathbb{Z}/d\mathbb{Z}$ satisfying $\sum_{i}w_i=0$, that is, we put
\begin{align*}
W:=\{\boldsymbol{w}=(w_1,\ldots, w_n)\in \left(\mathbb{Z}/d\mathbb{Z}\right)^n \mid \sum_{i=1}^nw_i=0 \}.
\end{align*}
We put $t:=(q-1)/d$. Then it is known that the number of $\mathbb{F}_q$-rational points on the projective diagonal hypersurface
\begin{align*}
x_1^d+\cdots +x_n^d=0
\end{align*}
is given by $\sum_{\boldsymbol{w}\in W}N_q(0,\boldsymbol{w})$, where
\footnotesize
\begin{equation}\label{hypersurface}
N_q(0,\boldsymbol{w}):=\begin{cases}
(q^{n-1}-1)/(q-1) & (w_i=0 \ \text{for all}\ i)\\
(1/q)\prod_{i=1}^{n}g(\omega^{w_it}) & (w_i \neq 0\ \text{for any}\ i)\\
0 &(\text{otherwise}).
\end{cases}
\end{equation}
\normalsize
(For example, see \cite[(2.12)]{Koblitz} and \cite{Weil}.)
We define an equivalence relation $\sim$ on $W$ by
\begin{align*}
\boldsymbol{w}\sim \boldsymbol{w'} \quad \text{if} \ \boldsymbol{w}-\boldsymbol{w'} \ \text{is a multiple of} \ (h_1,\ldots, h_n).
\end{align*}
We denote an equivalence class of $\boldsymbol{w}$ by $[\boldsymbol{w}]$. Then, we have the following theorem.
\begin{theorem}[{\cite[Theorem 2]{Koblitz}}]\label{Kobbformula}
We put $t:=(q-1)/d.$ For $\lambda\in \mathbb{F}_q$ with $\lambda\neq 0$ and $\lambda^d\neq (h_1^{h_1}\cdots h_n^{h_n})^{-1}$, we have
\footnotesize
\begin{equation}\label{Kobblitz}
\#D_\lambda(\mathbb{F}_q)=\sum_{\boldsymbol{w}\in W}N_q(0,\boldsymbol{w})+\frac{1}{q-1}\sum_{[\boldsymbol{w}]\in W/\sim}\sum_{j=0}^{q-2} \frac{\prod_{i=1}^ng(\omega^{w_it+h_ij})}{g(\omega^{dj})}\omega^{dj}(d\lambda).
\end{equation}
\normalsize
%where the summation is over $w\in W$ and over $s\in(\frac{d}{q-1}\mathbb{Z})/\mathbb{Z}$
\end{theorem}
\begin{rem}
The assumptions that $\lambda$ is not equal to zero and $\lambda^d$ is not equal to $(h_1^{h_1}\cdots h_n^{h_n})^{-1}$ imply that $D_{\lambda}$ is smooth.
\end{rem}
\begin{rem}
Note that $g(\omega^{w_it+h_ij})$ itself is not well-defined, and only $\prod_{i=1}^ng(\omega^{w_it+h_ij})$ is well-defined since we assume that $\sum_i w_i$ and $\sum_i h_i$ are equal to zero in $\mathbb{Z}/d\mathbb{Z}$.
\end{rem}
\begin{rem}
The notation in Theorem \ref{Kobbformula} differs from that in \cite[Theorem 2]{Koblitz}. In \cite[Theorem 2]{Koblitz}, the summation of the second term is over $s\in (d/(q-1))\mathbb{Z}/\mathbb{Z}$ and $w\in W$. We obtain an expression that is equivalent to the identity (\ref{Kobblitz}) by replacing $s$ with $ds$ and summing over $s\in (1/(q-1))\mathbb{Z}/\mathbb{Z}$.
\end{rem}

\subsection{An application of Koblitz's formula}\label{keisann}
We explain our strategy for proving the main theorem. 
As a first step, we apply Koblitz's formula for the Dwork hypersurfaces of degree six and we list up all of the elements of\ $W/\sim$. Next, we calculate the right hand side of the identity (\ref{Koblitz}), which is given in this subsection.
From $\sum_{\boldsymbol{w}\in W}N_q(0,\boldsymbol{w})=\sum_{[\boldsymbol{w}]\in W/\sim}\sum_{\boldsymbol{w}\in [\boldsymbol{w}]}N_q(0,\boldsymbol{w})$, we calculate the right hand side of the identity (\ref{Koblitz}) for each $[\boldsymbol{w}]\in W/\sim$.

By (\ref{Kobblitz}) with $d=n=6$ and $(h_1,\ldots, h_n)=(1,1,1,1,1,1)$, we have the following.
\begin{cor}\label{Kobformula}
For $\lambda\in \mathbb{F}_q$ with $\lambda\neq 0$ and $\lambda^6\neq 1$, we have
\footnotesize
\begin{equation}\label{Koblitz}
\#X_{\lambda}^6(\mathbb{F}_q)=\sum_{\boldsymbol{w}\in W}N_q(0,\boldsymbol{w})+\frac{1}{q-1}\sum_{[\boldsymbol{w}]\in W/\sim}\sum_{j=0}^{q-2}\frac{\prod_{i=1}^6g(\omega^{w_it+j})}{g(\omega^{6j})}\omega^{6j}(6\lambda).
\end{equation}
\normalsize
\end{cor}
Let $S_6$ be the symmetric group of degree six. An action of $S_6$ on $W/\sim$ is naturally defined.
\begin{defn}
%If the order of the orbit of $[\boldsymbol{w}]=[w_1, \ldots, w_6]$ is $k$, we let 
We let
\begin{equation*}
\langle[w_1, \cdots, w_6]\rangle^k=\left\{[x_1, \ldots, x_6]\in W/\sim \ \mid 
\begin{split}
\text{There exists}\ \sigma \in S_6\ \text{such that}\\
[x_{\sigma(1)},\ldots, x_{\sigma(6)}]=[w_1,\ldots,w_6]
\end{split}
\right\},
\end{equation*}
where 
\begin{equation*}
k=\#\left\{[x_1, \ldots, x_6]\in W/\sim \ \mid
\begin{split}
\text{There exists}\ \sigma \in S_6\ \text{such that} \\
[x_{\sigma(1)},\ldots, x_{\sigma(6)}]=[w_1,\ldots,w_6]
\end{split}
\right\}.
\end{equation*}
\end{defn}
By abuse of notation, we denote $\overline{a}=a$ for $\overline{a}\in \mathbb{Z}/6\mathbb{Z}$. Then, we have
\begin{align*}
\begin{split}
W/\sim \ &=\langle[0,0,0,0,0,0]\rangle^1\cup\langle[0,0,0,0,1,5]\rangle^{30}\cup\langle[0,0,0,0,2,4]\rangle^{30} \\
&\cup\langle[0,0,0,0,3,3]\rangle^{15}\cup\langle[0,0,0,1,1,4]\rangle^{60}\cup\langle[0,0,0,1,2,3]\rangle^{120} \\
&\cup\langle[0,0,0,2,2,2]\rangle^{20}\cup\langle[0,0,0,2,5,5]\rangle^{60}\cup\langle[0,0,0,3,4,5]\rangle^{120} \\
&\cup\langle[0,0,1,1,2,2]\rangle^{90}\cup\langle[0,0,2,2,4,4]\rangle^{30}\cup\langle[0,0,2,2,3,5]\rangle^{180} \\
&\cup\langle[0,0,1,3,3,5]\rangle^{180}\cup\langle[0,0,1,2,4,5]\rangle^{360}.
\end{split}
\end{align*}
We put 
\begin{equation*}
S_{[\boldsymbol{w}]}:=(q-1)^{-1}\sum_{j=0}^{q-2}\left(\prod_{i=1}^6g(\omega^{w_it+j})/g(\omega^{6j})\right)\omega^{6j}(6\lambda).
\end{equation*}

\begin{prop}\label{saisyo}
We obtain the following identities:
\small
\begin{equation}\label{[0,0,0,0,0,0]}
\begin{split}
\sum_{\boldsymbol{w}\in [0,0,0,0,0,0]}&N_q(0,\boldsymbol{w})+S_{[0,0,0,0,0,0]}\\
&=\frac{q^5-1}{q-1}+ q^4 \cdot {}_5F_4
\left(
\begin{array}{ccccc|c}
\omega_6 & \omega_3 & \omega_2 & \overline{\omega}_3 & \overline{\omega}_6\\
& \epsilon & \epsilon & \epsilon & \epsilon \\
\end{array}
\ \frac{1}{\lambda^6} \right)_q
\end{split}
\end{equation}

\begin{equation}\label{[0,0,0,0,1,5]}
\begin{split}
\sum_{\boldsymbol{w}\in [0,0,0,0,1,5]}&N_q(0,\boldsymbol{w})+S_{[0,0,0,0,1,5]}=q^3 \cdot \omega_6(-1)\cdot {}_3F_2
\left(
\begin{array}{ccc|c}
\omega_3 & \omega_2 & \overline{\omega}_3\\
& \epsilon & \epsilon \\
\end{array}
\ \frac{1}{\lambda^6} \right)_q
\end{split}
\end{equation}

\begin{equation}\label{[0,0,0,0,2,4]}
\begin{split}
\sum_{\boldsymbol{w}\in [0,0,0,0,2,4]}&N_q(0,\boldsymbol{w})+S_{[0,0,0,0,2,4]}=q^3\cdot {}_3F_2
\left(
\begin{array}{ccc|c}
\omega_6 & \omega_2 & \overline{\omega}_6 \\
& \epsilon & \epsilon \\
\end{array}
\ \frac{1}{\lambda^6} \right)_q
\end{split}
\end{equation}

\begin{equation}\label{[0,0,1,1,2,2]}
\begin{split}
\sum_{\boldsymbol{w}\in [0,0,1,1,2,2]}&N_q(0,\boldsymbol{w})+S_{[0,0,1,1,2,2]}=q^3\cdot {}_3F_2
\left(
\begin{array}{ccc|c}
\omega_2 & \overline{\omega}_3 & \overline{\omega}_6 \\
& \omega_6 & \omega_3 \\
\end{array}
\ \frac{1}{\lambda^6} \right)_q 
\end{split}
\end{equation}
\normalsize
\end{prop}

\begin{proof}
We give the proof of the identity (\ref{[0,0,0,0,0,0]}). First, we calculate $S_{[0,0,0,0,0,0]}$. We obtain
\small
\begin{equation*}
\begin{split}
S_{[0,0,0,0,0,0]}&=\frac{1}{q-1}\sum_{j=0}^{q-2}\frac{g(\omega^j)^6}{g(\omega^{6j})}\omega^{6j}(6\lambda) \\
&=\frac{1}{q-1}\left\{\frac{(-1)^6}{-1}+\frac{\sum_{i=1}^5g(\omega^{it})^6}{-1}+\sum_{\substack{j=0\\ j\neq 0,t,\ldots, 5t}}\frac{g(\omega^j)^6g(\omega^{-6j})}{q}\omega^{6j}(6\lambda) \right\} \\
\end{split}
\end{equation*}
\begin{equation*}
\begin{split}
&=\frac{1}{q-1}\left\{-1-\sum_{i=1}^5g(\omega^{it})^6+\frac{1}{q}\sum_{j=0}^{q-2}g(\omega^j)^6g(\omega^{-6j})\omega^{6j}(6\lambda)\right. \\
&\left.\qquad -\frac{1}{q}(-1)^6\cdot(-1)-\frac{1}{q}\sum_{i=1}^5g(\omega^{it})^6\cdot(-1) \right\} \\
&=\frac{1}{q-1}\left\{-\frac{q-1}{q}-\frac{q-1}{q}\sum_{i=1}^5g(\omega^{it})^6+\frac{1}{q}\sum_{j=0}^{q-2}g(\omega^j)^6g(\omega^{-6j})\omega^{6j}(6\lambda) \right\} \\
&=-\frac{1}{q}\ -\sum_{\boldsymbol{w}\in [0,0,0,0,0,0]}N_q(0,\boldsymbol{w})+\frac{1}{q(q-1)}\sum_{j=0}^{q-2}g(\omega^j)^6g(\omega^{-6j})\omega^{6j}(6\lambda). \\
\end{split}
\end{equation*}
\normalsize
Second, we calculate the hypergeometric function. We have
\small
\begin{equation*}
\begin{split}
&{}_5F_4
\left(
\begin{array}{ccccc|c}
\omega_6 & \omega_3 & \omega_2 & \overline{\omega}_3 & \overline{\omega}_6 \\
& \epsilon & \epsilon & \epsilon & \epsilon \\
\end{array}
\ \frac{1}{\lambda^6} \right)_q \\
&=\frac{q}{q-1}\sum_{j=0}^{q-2}\frac{\omega^j(-1)}{q^5}J(\omega^{t+j},\overline{\omega^j})J(\omega^{2t+j},\overline{\omega^j})J(\omega^{3t+j},\overline{\omega^j})J(\omega^{4t+j},\overline{\omega^j})J(\omega^{5t+j},\overline{\omega^j})\omega^j\left(\frac{1}{\lambda^6}\right) \\
&=\frac{1}{q^4(q-1)}\sum_{j=0}^{q-2}\frac{g(\omega^{-j})^5\omega^j\left(\frac{-1}{\lambda^6}\right)\prod_{i=1}^5g(\omega^{it+j})}{\prod_{i=1}^5g(\omega^{it})} \\
\end{split}
\end{equation*}
\begin{equation*}
\begin{split}
&=\frac{1}{q^4(q-1)}\sum_{j=0}^{q-2}\frac{g(\omega^{6j})\omega^{-6j}(6)g(\omega^{-j})^5\omega^j\left(\frac{-1}{\lambda^6}\right)}{g(\omega^j)} \\
&=\frac{1}{q^4(q-1)}\left\{\frac{-1\cdot(-1)^5}{-1}+\sum_{j=1}^{q-2}\frac{g(\omega^{6j})g(\omega^{-j})^6\omega^{-6j}(-6\lambda)}{q\cdot \omega^j(-1)} \right\} \\
&=\frac{1}{q^4(q-1)}\left\{-1-\frac{1}{q}(-1)\cdot(-1)^6+\frac{1}{q}\sum_{j=0}^{q-2}g(\omega^{6j})g(\omega^{-j})^6\omega^{-6j}(6\lambda) \right\} \\
&=\frac{-1}{q^5}+\frac{1}{q^5(q-1)}\sum_{j=0}^{q-2}g(\omega^{6j})g(\omega^{-j})^6\omega^{-6j}(6\lambda) \\
&=\frac{-1}{q^5}+\frac{1}{q^5(q-1)}\sum_{j=0}^{q-2}g(\omega^{-6j})g(\omega^{j})^6\omega^{6j}(6\lambda).
\end{split}
\end{equation*}
\normalsize
Here, the first equality follows from the definition of the hypergeometric function. The second equality follows from the definition of the Jacobi sum. The third equality is obtained by using Corollary \ref{sekikousiki}. The fourth equality is obtained by using the identity $g(\omega^j)g(\omega^{-j})=q\cdot \omega^j(-1)$ for $j\neq 0$.

This completes the proof of the identity (\ref{[0,0,0,0,0,0]}). Similarly, we can show the identities (\ref{[0,0,0,0,1,5]}), (\ref{[0,0,0,0,2,4]}) and (\ref{[0,0,1,1,2,2]}).
\end{proof}

\begin{prop}\label{Jacobi coefficient}
We obtain the following identities:
\begin{equation}\label{[0,0,0,0,3,3]}
\begin{split}
\sum&_{\boldsymbol{w}\in [0,0,0,0,3,3]}N_q(0,\boldsymbol{w})+S_{[0,0,0,0,3,3]}\\
&=-q^3 \omega_6(-1)J(\omega_2, \overline{\omega}_3, \overline{\omega}_6)\cdot {}_4F_3
\left(
\begin{array}{cccc|c}
\omega_6 & \overline{\omega}_6 & \overline{\omega}_3 & \omega_3 \\
& \epsilon & \epsilon & \omega_2 \\
\end{array}
\ \frac{1}{\lambda^6} \right)_q 
\end{split}
\end{equation}

\begin{equation}\label{[0,0,2,2,4,4]}
\begin{split}
\sum_{\boldsymbol{w}\in [0,0,2,2,4,4]}&N_q(0,\boldsymbol{w})+S_{[0,0,2,2,4,4]}\\
&=-q^2J(\omega_6, \omega_6)J(\omega_6,\omega_3,\omega_2)\cdot {}_3F_2
\left(
\begin{array}{ccc|c}
\omega_6 & \omega_2 & \overline{\omega}_6 \\
& \omega_3 & \overline{\omega}_3 \\
\end{array}
\ \frac{1}{\lambda^6} \right)_q
\end{split}
\end{equation}

\begin{equation}\label{[0,0,0,2,2,2]}
\begin{split}
\sum_{\boldsymbol{w}\in [0,0,0,2,2,2]}&N_q(0,\boldsymbol{w})+S_{[0,0,0,2,2,2]} \\
&=q^3\omega_6(-1) J(\omega_6,\omega_3,\omega_2)\cdot {}_4F_3
\left(
\begin{array}{cccc|c}
\omega_6 & \omega_2 & \overline{\omega}_3 & \overline{\omega}_6 \\
& \epsilon & \omega_3 & \omega_3 \\
\end{array}
\ \frac{1}{\lambda^6} \right)_q
\end{split}
\end{equation}
\end{prop}
\begin{proof}
We give the proof of only the identity (\ref{[0,0,0,0,3,3]}) since we can show the identities (\ref{[0,0,2,2,4,4]}) and (\ref{[0,0,0,2,2,2]}) similarly.  From a calculation similar to $S_{[0,0,0,0,0,0]}$, we obtain
\small
\begin{equation*}
\begin{split}
S&_{[0,0,0,0,3,3]}=\frac{1}{q-1}\sum_{j=0}^{q-2}\frac{g(\omega^j)^4g(\omega^{3t+j})^2}{g(\omega^{6j})}\omega^{6j}(6\lambda) \\
&=\frac{1}{q-1}\cdot \frac{(-1)^4g(\omega^{3t})^2}{-1}+\frac{1}{q-1}\frac{\sum_{i=1}^5g(\omega^{it})^4g(\omega^{3+it})^2}{-1} \\
&+\frac{1}{q-1}\sum_{\substack{j=0\\ j\neq 0,t,\ldots, 5t}}^{q-2}\frac{g(\omega^j)^4g(\omega^{3t+j})^2g(\omega^{-6j})\omega^{6j}(6\lambda)}{q} \\
&=-\frac{1}{q-1}g(\omega^{3t})^2-\frac{1}{q-1}\sum_{i=1}^5g(\omega^{it})^4g(\omega^{3t+it})^2 \\
&+\frac{1}{q(q-1)}\left(-(-1)^4g(\omega^{3t})^2\cdot(-1)-\sum_{i=1}^5g(\omega^{it})^4g(\omega^{3t+it})^2\cdot(-1)\right. \\
&\left.+\sum_{j=0}^{q-2}g(\omega^j)^4g(\omega^{3t+j})^2g(\omega^{-6j})\omega^{6j}(6\lambda) \right) \\
&=-\frac{1}{q-1}g(\omega^{3t})^2-\frac{1}{q-1}\sum_{i=1}^5g(\omega^{it})^4g(\omega^{3t+it})^2+\frac{1}{q(q-1)}g(\omega^{3t})^2 \\
&+\frac{1}{q(q-1)}\sum_{i=1}^5g(\omega^{it})^4g(\omega^{3t+it})^2+\frac{1}{q(q-1)}\sum_{j=0}^{q-2}g(\omega^j)^4g(\omega^{3t+j})^2g(\omega^{-6j})\omega^{6j}(6\lambda) \\
&=-\frac{1}{q}g(\omega^{3t})^2-\frac{1}{q}\sum_{i=1}^5g(\omega^{it})^4g(\omega^{3t+it})^2+\frac{1}{q(q-1)}\sum_{j=0}^{q-2}g(\omega^j)^4g(\omega^{3t+j})^2g(\omega^{-6j})\omega^{6j}(6\lambda) \\
&=-\omega^t(-1)-q-N_q(0,[0,0,0,0,3,3]) 
+\frac{1}{q(q-1)}\sum_{j=0}^{q-2}g(\omega^j)^4g(\omega^{3t+j})^2g(\omega^{-6j})\omega^{6j}(6\lambda).\\
\end{split}
\end{equation*}
\normalsize
Next, we calculate the hypergeometric function. By a similar way to the proof of the identity (\ref{[0,0,0,0,0,0]}) in Proposition \ref{saisyo}, we obtain
\small
\begin{equation*}
\begin{split}
{}&_4F_3
\left(
\begin{array}{cccc|c}
\omega_6 & \overline{\omega}_6 & \overline{\omega}_3 & \omega_3 \\
& \epsilon & \epsilon & \omega_2 \\
\end{array}
\ \frac{1}{\lambda^6} \right)_q \\
&=\frac{q}{q-1}\sum_{j=0}^{q-2}\frac{\omega^t(-1)}{q^4}J(\omega^{t+j},\overline{\omega^j})J(\omega^{5t+j},\overline{\omega^j})J(\omega^{4t+j},\overline{\omega^j})J(\omega^{2t+j},\overline{\omega^{3t+j}})\omega^j\left(\frac{1}{\lambda^6}\right) \\
\end{split}
\end{equation*}
\begin{equation*}
\begin{split}
&=\frac{\omega^t(-1)}{q^3(q-1)}\sum_{j=0}^{q-2}\frac{g(\omega^{t+j})g(\omega^{5t+j})g(\omega^{4t+j})g(\omega^{2t+j})g(\overline{\omega^j})^3g(\overline{\omega^{3t+j}})\omega^{-6j}(\lambda)}{g(\omega^t)g(\omega^{5t})g(\omega^{4t})g(\omega^{-t})} \\
&=\frac{\omega^t(-1)J(\omega^{2t},\omega^{3t})}{q^3(q-1)}\sum_{j=0}^{q-2}\frac{g(\omega^{6j})g(
\omega^{-j})^3g(\omega^{-3t-j})\omega^{-6j}(6\lambda)}{g(\omega^j)g(\omega^{3t+j})} \\
\end{split}
\end{equation*}
\begin{equation*}
\begin{split}
&=\frac{\omega_6(-1)J(\omega_3,\omega_2)}{q^3(q-1)}\left(\frac{(-1)(-1)^3g(\omega^{3t})}{(-1)g(\omega^{3t})}+\frac{(-1)g(\omega^{-3t})^3(-1)}{g(\omega^{3t})(-1)}\right. \\
&\left.+\sum_{\substack{j=0\\ j\neq 0,3t}}^{q-2}\frac{g(\omega^{6j})g(\omega^{-j})^4g(\omega^{-3t-j})^2\omega^{-6j}(6\lambda)}{q^2\omega^{-j-3t-j}(-1)}\right) \\
&=\frac{\omega_6(-1)J(\omega_3,\omega_2)}{q^3(q-1)}\left(-1-g(\omega^{3t})^2+\frac{1}{q}+\frac{g(\omega^{3t})^2}{q}+\sum_{j=0}^{q-2}\frac{g(\omega^{6j})g(\omega^{-j})^4g(\omega^{-3t-j})^2\omega^{-6j}(6\lambda)}{q^2\omega^t(-1)} \right) \\
&=\frac{\omega_6(-1)J(\omega_3,\omega_2)}{q^3(q-1)}\left(\frac{-(q-1)}{q}-\frac{(q-1)g(\omega^{3t})^2}{q}+\sum_{j=0}^{q-2}\frac{g(\omega^{6j})g(\omega^{-j})^4g(\omega^{-3t-j})^2\omega^{-6j}(6\lambda)}{q^2\omega^t(-1)} \right) \\
&=\frac{\omega_6(-1)J(\omega_3,\omega_2)}{q^3}\left(\frac{-1}{q}-\frac{g(\omega^{3t})^2}{q}\right)+\frac{J(\omega_3,\omega_2)}{q^5(q-1)}\sum_{j=0}^{q-2}g(\omega^{6j})g(\omega^{-j})^4g(\omega^{-3t-j})^2\omega^{-6j}(6\lambda) \\
&=\frac{J(\omega_3,\omega_2)}{q^3}\left(\frac{-\omega_6(-1)}{q}-1 \right)+\frac{J(\omega_3,\omega_2)}{q^5(q-1)}\sum_{j=0}^{q-2}g(\omega^{6j})g(\omega^{-j})^4g(\omega^{-3t-j})^2\omega^{-6j}(6\lambda). \\
\end{split}
\end{equation*}
\normalsize
Then, the identity (\ref{[0,0,0,0,3,3]}) follows from the identity $\frac{q^4}{J(\omega_3,\omega_2)}=-q^3\omega_6(-1)J(\omega_2,\overline{\omega}_3,\overline{\omega}_6)$.
% Then, the identity $\frac{q^4}{J(\omega_3,\omega_2)}=-q^3\omega_6(-1)J(\omega_2,\overline{\omega}_3,\overline{\omega}_6)$ shows the identity (\ref{[0,0,0,0,3,3]}).
\end{proof}

\begin{prop}\label{60}
We obtain
\begin{equation*}
\begin{split}
\sum_{\boldsymbol{w}\in [0,0,0,1,1,4]}&N_q(0,\boldsymbol{w})+S_{[0,0,0,1,1,4]}+\sum_{\boldsymbol{w}\in [0,0,0,2,5,5]}N_q(0,\boldsymbol{w})+S_{[0,0,0,2,5,5]}\\
&=q^2 \omega_6(-1)J(\omega_6, \omega_6, \overline{\omega}_3)J(\omega_2,\overline{\omega}_3,\overline{\omega}_6)\cdot {}_3F_2
\left(
\begin{array}{ccc|c}
\omega_6 & \overline{\omega}_3 & \omega_2 \\
& \epsilon & \overline{\omega}_6 \\
\end{array}
\ \frac{1}{\lambda^6} \right)_q\\
&\qquad+q^2 J(\omega_3,\omega_3,\omega_3)J(\omega_2, \overline{\omega_3}, \overline{\omega}_6)\cdot {}_3F_2
\left(
\begin{array}{ccc|c}
\omega_3 & \overline{\omega}_6 & \omega_2 \\
& \epsilon & \omega_6 \\
\end{array}
\ \frac{1}{\lambda^6} \right)_q .
\end{split}
\end{equation*}
\end{prop}
\begin{proof}
By an argument similar to that in Proposition \ref{Jacobi coefficient}, we have
\small
\begin{equation*}
\begin{split}
\sum_{\boldsymbol{w}\in [0,0,0,1,1,4]}&N_q(0,\boldsymbol{w})+S_{[0,0,0,1,1,4]} \\
&=q^2 \cdot J(\omega_3,\omega_3,\omega_3)J(\omega_2, \overline{\omega_3}, \overline{\omega}_6)\cdot {}_3F_2
\left(
\begin{array}{ccc|c}
\omega_3 & \overline{\omega}_6 & \omega_2 \\
& \epsilon & \omega_6 \\
\end{array}
\ \frac{1}{\lambda^6} \right)_q \\
&+\frac{1}{q}g(\omega^t)^2g(\omega^{4t})-\frac{1}{q}g(\omega^{5t})^3g(\omega^{3t})+\omega^t(-1)g(\omega^{2t})^3\\
&-\frac{1}{q}g(\omega^{2t})g(\omega^{5t})^2+\frac{1}{q}g(\omega^t)^3g(\omega^{3t})-\omega^t(-1)g(\omega^{4t})^3.
\end{split}
\end{equation*}
and

\begin{equation*}
\begin{split}
\sum_{\boldsymbol{w}\in [0,0,0,2,5,5]}&N_q(0,\boldsymbol{w})+S_{[0,0,0,2,5,5]}\\
&=q^2 \omega_6(-1)J(\omega_6, \omega_6, \overline{\omega}_3)J(\omega_2,\overline{\omega}_3,\overline{\omega}_6)\cdot {}_3F_2
\left(
\begin{array}{ccc|c}
\omega_6 & \overline{\omega}_3 & \omega_2 \\
& \epsilon & \overline{\omega}_6 \\
\end{array}
\ \frac{1}{\lambda^6} \right)_q\\
&+\frac{1}{q}g(\omega^{5t})^2g(\omega^{2t})+\omega^t(-1)g(\omega^{4t})^3-\frac{1}{q}g(\omega^t)^3g(\omega^{3t}) \\
&-\frac{1}{q}g(\omega^t)^2g(\omega^{4t})-\omega^t(-1)g(\omega^{2t})^3+\frac{1}{q}g(\omega^{5t})^3g(\omega^{3t}).
\end{split}
\end{equation*}
This completes the proof of Proposition \ref{60}.
\end{proof}

In the same fashion as the proof of \cite[Proposition 4.6]{case4}, we can obtain the following results by using Lemma \ref{sekikousiki}.
\begin{prop}
We obtain the following identities:
\begin{equation}
\begin{split}
\sum_{\boldsymbol{w}\in [0,0,0,1,2,3]}&N_q(0,\boldsymbol{w})+S_{[0,0,0,1,2,3]}=-q^2J(\omega_6, \omega_3,\omega_2)\cdot {}_2F_1
\left(
\begin{array}{cc|c}
\omega_6 & \omega_3 \\
& \epsilon \\
\end{array}
\ \frac{1}{\lambda^6} \right)_q 
\end{split}
\end{equation}

\begin{equation}
\begin{split}
\sum_{\boldsymbol{w}\in [0,0,0,3,4,5]}&N_q(0,\boldsymbol{w})+S_{[0,0,0,3,4,5]}=-q^2 J(\omega_2, \overline{\omega}_3,\overline{\omega}_6)\cdot {}_2F_1
\left(
\begin{array}{cc|c}
\overline{\omega}_3 & \overline{\omega}_6 \\
& \epsilon \\
\end{array}
\ \frac{1}{\lambda^6} \right)_q 
\end{split}
\end{equation}

\begin{equation}
\begin{split}
\sum_{\boldsymbol{w}\in [0,0,1,3,3,5]}&N_q(0,\boldsymbol{w})+S_{[0,0,1,3,3,5]}=
-q^2J(\omega_6, \omega_3,\omega_2)\cdot {}_2F_1
\left(
\begin{array}{cc|c}
\omega_3 & \overline{\omega}_3 \\
& \omega_2 \\
\end{array}
\ \frac{1}{\lambda^6} \right)_q
\end{split}
\end{equation}

\begin{equation}
\begin{split}
\sum_{\boldsymbol{w}\in [0,0,2,2,3,5]}&N_q(0,\boldsymbol{w})+S_{[0,0,2,2,3,5]} =
%\end{equation*}
%\end{prop}
-q^2J(\omega_6,\omega_3,\omega_2)\cdot {}_2F_1
\left(
\begin{array}{cc|c}
\omega_3 & \overline{\omega}_6 \\
& \overline{\omega}_3 \\
\end{array}
\ \frac{1}{\lambda^6} \right)_q
\end{split}
\end{equation}

\end{prop}

\begin{proof}
In the same way as the proof of \cite[Proposition 4.6]{case4}, we have
\small
\begin{equation*}
\begin{split}
\sum_{\boldsymbol{w}\in [0,0,0,1,2,3]}&N_q(0,\boldsymbol{w})+S_{[0,0,0,1,2,3]}\\
&=-q^2\omega^t(-1)J(\omega^t, \omega^{2t},\omega^{3t})\cdot\frac{\omega^t(-1)}{q}\sum_{y\in \mathbb{F}_q}\omega^{2t}(y)\omega^{5t}(1-y)\omega^{4t}(1-\lambda^6y) \\
&\left(=-q^2\omega_6(-1)J(\omega_6, \omega_3,\omega_2)\cdot {}_2F_1
\left(
\begin{array}{cc|c}
\omega_3 & \omega_3 \\
& \omega_6 \\
\end{array}
\ \lambda^6 \right)_q \right).
\end{split}
\end{equation*}
\normalsize
Then, the substitution $y\mapsto y/\lambda^6$ implies
\begin{equation*}
\begin{split}
&\sum_{\boldsymbol{w}\in [0,0,0,1,2,3]}N_q(0,\boldsymbol{w})+S_{[0,0,0,1,2,3]}\\
&=-q^2\omega^t(-1)J(\omega^t, \omega^{2t},\omega^{3t})\cdot\frac{\omega^t(-1)}{q}\sum_{y\in \mathbb{F}_q}\omega^{2t}\left(\frac{y}{\lambda^6}\right)\omega^{5t}\left(1-\frac{y}{\lambda^6}\right)\omega^{4t}(1-y) \\
&=-q^2\omega^t(-1)J(\omega^t, \omega^{2t},\omega^{3t})\cdot\frac{\omega^t(-1)}{q}\sum_{y\in \mathbb{F}_q}\omega^{2t}(y)\omega^{4t}(1-y)\omega^{5t}\left(1-\frac{y}{\lambda^6}\right) \\
&=-q^2J(\omega_6, \omega_3,\omega_2)\cdot {}_2F_1
\left(
\begin{array}{cc|c}
\omega_6 & \omega_3 \\
& \epsilon \\
\end{array}
\ \frac{1}{\lambda^6} \right)_q .
\end{split}
\end{equation*}
\end{proof}

Finally, we calculate in case of $[\boldsymbol{w}]=[0,0,1,2,4,5]$. 
\begin{prop}\label{saigo}
We obtain
\begin{equation*}
\begin{split}
\sum_{\boldsymbol{w}\in [0,0,1,2,4,5]}&N_q(0,\boldsymbol{w})+S_{[0,0,1,2,4,5]}=q^2\omega_2(1-\lambda^6).
\end{split}
\end{equation*}
\end{prop}
\begin{proof}
First, we obtain
\begin{equation*}
\begin{split}
\sum_{\boldsymbol{w}\in [0,0,1,2,4,5]}N_q(0,\boldsymbol{w})&=N_q(0,(3,3,4,5,1,2))\\
&=\frac{1}{q}g(\omega^{3t})^2g(\omega^{4t})g(\omega^{5t})g(\omega^t)g(\omega^{2t})=q^2.
\end{split}
\end{equation*}
Second, we calculate $S_{[0,0,1,2,4,5]}$. From Corollary \ref{sekikousiki}, we have
\begin{equation*}
\begin{split}
S_{[0,0,1,2,4,5]}&=\frac{1}{q-1}\prod_{k=1}^5g(\omega^{kt})\sum_{j=0}^{q-2}\frac{g(\omega^j)^2g(\omega^{t+j})g(\omega^{2t+j})g(\omega^{4t+j})g(\omega^{5t+j})}{g(\omega^j)\cdots g(\omega^{5t+j})}\omega^{6j}(\lambda)\\
&=\frac{q^2\omega^t(-1)g(\omega^{3t})}{q-1}\sum_{j=0}^{q-2}\frac{g(\omega^j)}{g(\omega^{3t+j})}\omega^{6j}(\lambda).
\end{split}
\end{equation*}
We consider $g(\omega^j)/g(\omega^{3t+j})$. If $j\neq 3t$, we have
\begin{equation*}
\begin{split}
\frac{g(\omega^j)}{g(\omega^{3t+j})}&=\frac{g(\omega^j)g(\omega^{-3t-j})}{g(\omega^{3t+j})g(\omega^{-3t-j})} =\frac{1}{q}\ \omega^{t+j}(-1)g(\omega^j)g(\omega^{-3t-j}).
\end{split}
\end{equation*}
If $j=3t$, we have
\begin{equation*}
\frac{g(\omega^j)}{g(\omega^{3t+j})}\omega^{6j}(\lambda)=\frac{g(\omega^{3t})}{-1}\cdot 1=-g(\omega^{3t}).
\end{equation*}
By Lemma \ref{turai}, we obtain
\footnotesize
\begin{equation*}
\begin{split}
\sum_{j=0}^{q-2}&\frac{g(\omega^j)}{g(\omega^{3t+j})}\omega^{6j}(\lambda) \\
&=\sum_{j=0,j\neq 3t}^{q-2}\frac{g(\omega^j)}{g(\omega^{3t+j})}\omega^{6j}(\lambda)+\frac{g(\omega^{3t})}{g(\omega^{3t+3t})}\omega^{6\cdot3t}(\lambda) \\
&=\frac{1}{q}\sum_{j=0, j\neq 3t}^{q-2}\omega^{t+j}(-1)g(\omega^j)g(\omega^{-3t-j})\omega^{6j}(\lambda) -g(\omega^{3t}) \\
&=\frac{1}{q}\left\{\sum_{j=0}^{q-2}\omega^{t+j}(-1)g(\omega^j)g(\omega^{-3t-j})\omega^{6j}(\lambda)-\omega^{t+3t}(-1)g(\omega^{3t})g(\omega^{-3t-3t})\omega^{6\cdot 3t}(\lambda)\right\}-g(\omega^{3t}) \\
&=\frac{1}{q}\{\omega^t(-1)\cdot(q-1)g(\omega^{-3t})\omega^{-3t}(-1)\omega^{3t}(1-\lambda^6)-g(\omega^{3t})\cdot (-1)\}-g(\omega^{3t}) \\
&=\frac{1}{q}\{(q-1)g(\omega^{3t})\omega^{3t}(1-\lambda^6)+g(\omega^{3t})\}-g(\omega^{3t}).
\end{split}
\end{equation*}
\normalsize
Hence, we conclude
\footnotesize
\begin{equation*}
\begin{split}
\sum_{\boldsymbol{w}\in [0,0,1,2,4,5]}&N_q(0,\boldsymbol{w})+S_{[0,,0,1,2,4,5]} \\
&=q^2+\frac{1}{q-1}q^2\omega^t(-1)g(\omega^{3t})
\left[\frac{1}{q}\{(q-1)g(\omega^{3t})\omega^{3t}(1-\lambda^6)+g(\omega^{3t})\}-g(\omega^{3t})\right] \\
&=q^2+q^2\omega^{3t}(1-\lambda^6)+\frac{q^2-q^3}{q-1} \\
&=q^2\omega^{3t}(1-\lambda^6).
\end{split}
\end{equation*}
\normalsize
\end{proof}
%\subsection{The proof of Theorem\ref{main theorem}}
Let us prove Theorem \ref{main theorem}.
\begin{proof}
From Equation (\ref{Koblitz}), we obtain
\footnotesize
\begin{equation*}
\begin{split}
\#X_{\lambda}^6(\mathbb{F}_q)&=\sum_{\boldsymbol{w}\in W}N_q(0,\boldsymbol{w})+\frac{1}{q-1}\sum_{[\boldsymbol{w}]\in W/\sim}\sum_{j=0}^{q-2}\frac{\prod_{i=1}^6g(\omega^{w_it+j})}{g(\omega^{6j})}\omega^{6j}(6\lambda) \\
&=\sum_{[\boldsymbol{w}]\in W/\sim}\sum_{\boldsymbol{w}\in [w]}N_q(0,\boldsymbol{w})+\frac{1}{q-1}\sum_{[\boldsymbol{w}]\in W/\sim}\sum_{j=0}^{q-2}\frac{\prod_{i=1}^6g(\omega^{w_it+j})}{g(\omega^{6j})}\omega^{6j}(6\lambda) \\
&=\sum_{[\boldsymbol{w}]\in W/\sim} \left\{\sum_{\boldsymbol{w}\in [\boldsymbol{w}]}N_q(0,\boldsymbol{w})+\frac{1}{q-1}\sum_{j=0}^{q-2}\frac{\prod_{i=1}^6g(\omega^{w_it+j})}{g(\omega^{6j})}\omega^{6j}(6\lambda)\right\} \\
&=\sum_{[\boldsymbol{w}]\in W/\sim} \left\{\sum_{\boldsymbol{w}\in [\boldsymbol{w}]}N_q(0,\boldsymbol{w})+S_{[\boldsymbol{w}]}\right\}. \\
\end{split}
\end{equation*}
\normalsize
Then, Propositions \ref{saisyo} through \ref{saigo} complete the proof of Theorem \ref{main theorem}. 
\end{proof}
\normalsize

\appendix
\section{The proof by Miyatani's formula}
In this appendix, we give the proof of Theorem \ref{main theorem} by Miyatani's formula which is expressed in terms of McCarthy's finite-field hypergeometric functions. In \cite[Proposition 3.9]{miyatani}, Miyatani expressed the number of rational points on 
hypersurfaces$$X_\lambda : c_1X^{a_1}+\cdots+c_{n+1}X^{a_{n+1}}=\lambda X_1\cdots X_{n+1},$$
where $\lambda \in \mathbb{F}_q^\times$ such that $X_{\lambda}$ is smooth, $c_1,\ldots, c_{n+1}\in \mathbb{F}_q^\times$, and $a_i:={}^t(a_{1,i},\ldots,a_{n+1,i})\in \mathbb{Z}_{\geq 0}^{n+1}$ with $a_{1,i}+\cdots +a_{n+1,i}=n+1$ and none of $a_i$'s being equal to $^t(1,\ldots,1)$ (for $i=1,\ldots,n+1$). Note that the notation $X^{a_i}$ means the monomial $X_1^{a_{1,i}}\cdots X_{n+1}^{a_{n+1,i}}$ for $a_i={}^t(a_{1,i},\ldots,a_{n+1,i})$.
In our case, McCarthy's hypergeometric function can be expressed as a product of the normalized Jacobi sums and Greene's hypergeometric function. $($See Proposition \ref{transformation}.$)$ Thus, we can also obtain Theorem \ref{main theorem}, {\cite[Theorem 1.1]{case4}} and {\cite[Theorem 1.2]{cased}} by Miyatani's formula. 

\subsection{McCarthy's finite-field hypergeometric functions}
In this subsection, we introduce the finite-field hypergeometric function defined by McCarthy in \cite{transformation}.

For $A_1,\ldots, A_{n+1}$, $B_1, \ldots, B_{n+1} \in \widehat{\mathbb{F}}_q^\times$, we define McCarthy's finite-field hypergeometric function ${}_{n+1}\widetilde{F}_{n+1}$ by 
\footnotesize
\[{}_{n+1}\widetilde{F}_{n+1}
\left(
\begin{array}{cccc}A_1 & \cdots & A_{n+1} \\
B_1& \cdots& B_{n+1}  \\
\end{array}
;x \right)_{\mathbb{F}_q}:=\frac{-1}{q-1}\sum_{\chi\in \widehat{\mathbb{F}}_q^\times}\prod_{i=1}^{n+1}\frac{g(A_i\chi)}{g(A_i)}\frac{g(\overline{B_i\chi})}{g(\overline{B_i})}\chi(-1)^{n+1}\chi(x).
\]
\normalsize
%\end{defn}}
Furthermore, we use the following notation.
\begin{defn}
Let $A_1,\ldots, A_{n+1}, B_1, \ldots, B_{n+1}$ be characters on $\mathbb{F}_q^\times$ in $\mathbb{C}^\times$. By sorting index, we assume that $\{A_1, \ldots, A_{n'+1}\}$ and $\{B_1,\ldots, B_{n'+1}\}$ have no intersection and that $\{A_{n'+2},\ldots, A_{n+1}\}$ and $\{B_{n'+2},\ldots, B_{n+1}\}$ are equal as the multiset. Then we define the hypergeometric function with reduced parameters over $\mathbb{F}_q$ by
\[
{}_\bullet\widetilde{F}_\bullet \mathrm{Red}
\left(
\begin{array}{cccc}A_1 & \cdots & A_{n+1} \\
B_1& \cdots& B_{n+1}  \\
\end{array}
;x \right)_{\mathbb{F}_q}:={}_{n'+1}\widetilde{F}_{n'+1}
\left(
\begin{array}{cccc}A_1 & \cdots & A_{n'+1} \\
B_1& \cdots& B_{n'+1}  \\
\end{array}
;x \right)_{\mathbb{F}_q}.
\]
\end{defn}

\subsection{Miyatani's formula}\label{notations}
In this subsection, we recall Miyatani's formula. To state his formula, we introduce some notations.
For the matrix $A':=(a_{i,j}-1)_{1\leq i,j\leq n+1}$, the kernel $\Delta$ of the homomorphism $\mathbb{Z}^{n+1}\rightarrow \mathbb{Z}^{n+1}$ defined by $A'$ is generated by an uniquely determined vector $^t(\alpha_1\ldots, \alpha_{n+1})$ with all $\alpha_i >0$. (See \cite[Proposition 2.2]{miyatani}.) 
We put $\alpha:=\sum_{i=1}^{n+1}\alpha_i$. 
Let $N$ be a positive integer divisible by all $\alpha_i$ and $\alpha$, and we put $\mathbb{Z}_N:=\mathbb{Z}/N\mathbb{Z}.$ Let $$f_N : (\mathbb{Z}_N)^{n+1}/\Delta \rightarrow (\mathbb{Z}_N)^{n+1}$$
be the morphism induced by the endomorphism of $(\mathbb{Z}_N)^{n+1}$ defined by the matrix $A' \bmod N$ and let $d_1,\ldots, d_n$ be non-zero elementary divisors of $A'$. By an isomorphism $\mathrm{Im}(f_N)\simeq \bigoplus_{i=1}^nd_i\mathbb{Z}/N\mathbb{Z}$, we see that the kernel of $f_N$ consists of $d:=d_1\cdots d_n$ elements.  We fix $s_0:= {}^t(0,\ldots,0), s_1,\ldots,s_{d-1}\in \{0,\ldots,q-2\}^{n+1}$ that represent $\mathrm{Ker}(f_{q-1})$.
For $s_j:={^t(s_{1,j},\ldots,s_{n+1,j})}$, we put $|s_j|:=\sum_is_{i,j}$, $t_{i,j}:=s_{i,j}/\alpha_i$ and $t_j:=(\sum_is_{i,j})/\alpha$. To simplify notations, we put $\omega_{\beta}:=\omega^{\frac{q-1}{\beta}}$ for a positive integer $\beta$. Note that $\omega_{\beta}$ is well-defined if $q$ is congruent to $1$ modulo $\beta$.
For each $j=0,\ldots,d-1$, we put
\[
F(s_j) := \begin{cases}
{}_\bullet\widetilde{F}_\bullet \mathrm{Red}
\left(
\begin{array}{cccc} & [\omega_{\alpha}] &  \\
[\omega_{\alpha_1}]& \cdots &[\omega_{\alpha_{n+1}}] \\
\end{array}
;C\lambda^{-\alpha} \right)_{\mathbb{F}_q} & (j=0)
\\
\vspace{2mm}
q^{\delta_{s_j}-1}\cdot{}_\bullet\widetilde{F}_\bullet \mathrm{Red}
\left(
\begin{array}{cccc} & [\omega_{\alpha}] &  \\
\omega^{t_{1,j}}[\omega_{\alpha_1}]& \cdots &\omega^{t_{n+1,j}}[\omega_{\alpha_{n+1}}] \\
\end{array}
;C\lambda^{-\alpha} \right)_{\mathbb{F}_q} & (j\neq 0),
\end{cases}
\]
where $[\omega_\beta]$ and $\omega^{t_{i,j}}[\omega_{\beta}]$ are respectively the sequences $\varepsilon$, $\omega_{\beta}$, $\omega_{\beta}^2$, $\ldots$, $\omega_{\beta}^{\beta-1}$ and $\omega^{t_{i,j}}$, $\omega^{t_{i,j}}\cdot\omega_{\beta}$, $\omega^{t_{i,j}}\cdot\omega_{\beta}^2$, $\ldots$, $\omega^{t_{i,j}}\cdot\omega_{\beta}^{\beta-1}$, and where $C:=\alpha^{\alpha}\cdot\prod_{i=1}^{n+1}(c_i/\alpha_i)^{\alpha_i}$ and 
$$\delta_{s_j}:=\begin{cases}
1 &(|s_j| \equiv 0 \bmod q-1)
\vspace{2mm}
\\
0 &(|s_j| \not\equiv 0 \bmod q-1).
\end{cases}$$

For each $j=0,\cdots, d-1$, we put
\begin{equation*}
\begin{split}
\gamma(s_j):=\prod_{i=1}^{n+1}&\omega^{s_{i,j}}(\alpha_i^{-1}c_i)\omega^{|s_j|}((-\lambda)^{-1}\alpha)  \\
&\times \prod_{i=1}^{n+1}\left(g(\omega^{-t_{i,j}})\prod_{b_i=1}^{\alpha_i-1}\frac{g(\omega^{-t_{i,j}}\omega_{\alpha_i}^{b_i})}{g(\omega_{\alpha_i}^{b_i})}\right)g(\omega^{t_j})\prod_{b=1}^{\alpha-1}\frac{g(\omega^{t_j}\omega_{\alpha}^b)}{g(\omega_\alpha^b)}.
\end{split}
\end{equation*}

For the matrix $A:=(a_{i,j})$, we define $(n+2)\times(n+1)$ matrix $\widetilde{A}$ as $\widetilde{A}:=(\frac{A}{1\cdots 1})$. For a $k\times l$ matrix $M:=(m_{i,j})_{i,j}$ with coefficients in $\mathbb{Z}$, we define the morphism $\varphi(M): (\widehat{\mathbb{F}}_q^\times)^l \rightarrow (\widehat{\mathbb{F}}_q^\times)^k$ by 
$\varphi(M)((\chi_i)_{i=1,\ldots,l})=(\chi_1^{m_{j,1}}\cdots \chi_n^{m_{j,l}})_{j=1,\ldots,k}$. Let $J:=\{j_1,\ldots,j_t\}$ be an arbitrary subset of $\{1,\ldots,n+1\}$ with $t\geq (n+1)/2$ elements, let $\sigma(J)$ be the number of indices $i\in \{1,\ldots,n+1\}$ with $a_{i,j}=0$ for all $j\notin J$ and let $i_1,\ldots,i_{\sigma(J)}$ be all such indices.  $($We may assume that $i_1,\ldots,i_{\sigma(J)}$ are elements of $J$. See \cite[Proposition 2.1]{miyatani}.$)$ Then we put
\begin{equation*}
u:=\sum_{\substack{J\subset \{1,\ldots,n+1\}\\ \frac{n+1}{2}\leq \#J\leq n}}\sum_{i=0}^{\#J-\sigma(J)}(-1)^{\#J-\sigma(J)-i}q^{i-1}\sum\prod_{j=1}^{\sigma(J)}g(\chi_j^{-1})\chi_j(c_j),
\end{equation*}
where the most inner sum runs through all elements ${}^t(\chi_1,\ldots,\chi_{\sigma(J)}) \in \mathrm{Ker}(\varphi(\widetilde{A}))$ such that exactly $n-2i+1$ components are non-trivial.
Then  Miyatani's formula is the following. 
\begin{theorem}[{\cite[Proposition 3.9]{miyatani}}]\label{mmiyatani}
Suppose that the following conditions hold: 
\begin{enumerate}
\item $q-1$ is divisible by all $\alpha_i$'s and by $\alpha$. \label{assumption1}
\item Each $s_{i,j}$ is divisible by $\alpha_i$ and $|s_j|(=\sum_is_{i,j})$ is divisible by $\alpha$.\label{assumption2}
\item %Let $J:=\{j_1,\ldots,j_t\}$ be an arbitrary subset of $\{1,\ldots,n+1\}$ with $t\geq (n+1)/2$ elements, let $\sigma(J)$ be the number of indices $i\in \{1,\ldots,n+1\}$ with $a_{i,j}=0$ for all $j\notin J$ and let $i_1,\ldots,i_{\sigma(J)}$ be all such indices.  $($We may assume that $i_1,\ldots,i_{\sigma(J)}$ are elements of $J$. See \cite[Proposition 2.1]{miyatani}.$)$ Then
All elementary divisors of the $(t+1)\times \sigma(J)$ matrix 
$$\begin{pmatrix}
a_{j_1,i_1} & \cdots &a_{j_1,i_{\sigma(J)}} \\
 &\vdots &  \\
a_{j_t,i_1} & \cdots & a_{j_t,i_{\sigma(J)}}  \\
1& \cdots & 1 \\
\end{pmatrix}
$$
divide $q-1$.\label{assumption3}
\end{enumerate}
Then for $\lambda\in \mathbb{F}_q^{\times}$ such that $X_{\lambda}$ is smooth and $\lambda^{\alpha}\neq C(=\alpha^{\alpha}\prod_{i=1}^{n+1}(c_i/\alpha_i)^{\alpha_i})$, we have
\begin{equation*}
\#X_\lambda(\mathbb{F}_q)=\sum_{i=1}^{n-1}q^i+u+q^{(n-1)/2}\cdot D+(-1)^n\sum_{j=0}^{d-1}\gamma(s_j)F(s_j),
\end{equation*}
where D is defined to be the number of subsets $J\subset \{1,2,\ldots ,n+1\}$ such that $\#J$ is equal to $(n+1)/2$ and that for all $i=1,\ldots,n+1$ there exists $j\notin J$ such that $a_{i,j}\geq 1$. 
\end{theorem}

\subsection{The proof of the main theorem}
First, we give a lemma to apply Theorem \ref{mmiyatani} for the Dwork hypersurfaces. From the Hasse-Davenport product relation (Theorem \ref{hassedave}), we have the following.
\begin{lemma}\label{nontrivial}
We have
\begin{equation*}
\gamma(s_j)=g(\omega^{t_j\alpha})\omega^{-t_j\alpha}(\alpha)\prod_{i=1}^{n+1}\omega^{s_{i,j}}(\alpha_i^{-1}c_i)\omega^{|s_j|}((-\lambda)^{-1}\alpha)\prod_{i=1}^{n+1}g(\omega^{-t_{i,j}\alpha_i})\omega^{t_{i,j}\alpha_i}(\alpha_i).
%g(\omega^{|s|})\omega^{-|s|}(\alpha) \prod_{i=1}^6\omega^{s_i}(\alpha_i^{-1}c_i)\omega^{|s|}((-\lambda)^{-1}\alpha)\prod_{i=1}^6g(\omega^{-s_i})\omega^{s_i}(\alpha_i)
\end{equation*}
\end{lemma}

Next, we apply Miyatani's formula for the Dwork hypersurfaces of degree six. 
By definition, $u$ and $D$ in Theorem \ref{mmiyatani} are equal to zero, and the matrix $A'$ is of size $6\times 6$ and is given by

$$A'=\begin{pmatrix}
5 & -1 & \cdots &-1 \\
-1 & 5 & \ddots& \vdots \\
\vdots & \ddots & \ddots & -1  \\
-1& \cdots & -1 & 5 \\
\end{pmatrix}
.$$ 
$\Delta$ is the group generated by $^t(1,1,1,1,1,1)$. (See \cite[Example 2.5]{miyatani}.)
Then, we have
\footnotesize
\begin{equation*}
\begin{split}
\mathrm{Ker}(f_{q-1})&=\left\{{}^t(x_1,\ldots,x_6)\in (\mathbb{Z}_{q-1})^6/{\Delta}\mid 
\begin{array}{l}
6x_1=\cdots =6x_6, \\
x_1+\cdots+x_5-5x_6=0\ in\ \mathbb{Z}_{q-1}
\end{array}
\right\} \\
&=\left\{{}^t \left(w_1t\pmod{q-1},\ldots,w_6t\pmod{q-1}\right)\mid 
\begin{array}{l}
w_1,\ldots,w_6\in \mathbb{Z}, t=\frac{q-1}{6} \\
w_1t+\cdots + w_6t=0\ in\ \mathbb{Z}_{q-1}
\end{array}
\right\}.
\end{split}
\end{equation*}
\normalsize
Note that $\mathrm{Ker}(f_{q-1})$ has $6^4$ elements since the elementary divisors of $A'$ are $1, 6, 6, 6, 6, 0$. (See also \cite[Example 3.3]{miyatani}.)
From Theorem \ref{mmiyatani} and Lemma \ref{nontrivial}, we have the following.%Then, Miyatani gave the following formula.
\begin{cor}\label{miyatani}
Let $q=p^e$ be a power of a prime number such that $q$ is congruent to $1$ modulo $6$. For $\lambda \in \mathbb{F}_q$ with $\lambda\neq 0$ and $\lambda^6 \neq 1$, we have 
\begin{equation*}
\#X_{\lambda}^6(\mathbb{F}_q)=\frac{q^5-1}{q-1}-\sum_{s_j={}^t(s_{1,j},\ldots,s_{6,j})\in \{s_0,\ldots,s_{6^4-1}\}}\gamma(s_j)F(s_j),
\end{equation*}
where
\begin{equation*}
\gamma(s_j)=-\prod_{i=1}^6g(\omega^{-s_{i,j}})
\end{equation*}
and
\footnotesize
\begin{equation*}
\begin{split}
F(s_j)=\begin{cases}
{}_\bullet\widetilde{F}_\bullet \mathrm{Red}
\left(
\begin{array}{ccccccc}\epsilon & \omega_6 & \omega_3 & \omega_2 &\overline{\omega}_3 & \overline{\omega}_6 \\
\epsilon& \epsilon & \epsilon & \epsilon & \epsilon & \epsilon  \\
\end{array}
;\frac{1}{\lambda^6} \right)_{\mathbb{F}_q} &(j=0)
\\
{}_\bullet\widetilde{F}_\bullet \mathrm{Red}
\left(
\begin{array}{ccccccc}\omega^{\frac{|s_j|}{6}} & \omega^{t+\frac{|s_j|}{6}} & \omega^{2t+\frac{|s_j|}{6}} & \omega^{3t+\frac{|s_j|}{6}} &\omega^{4t+\frac{|s_j|}{6}} & \omega^{5t+\frac{|s_j|}{6}} \\
\omega^{s_{1,j}}& \omega^{s_{2,j}} & \omega^{s_{3,j}} & \omega^{s_{4,j}} & \omega^{s_{5,j}} & \omega^{s_{6,j}}  \\
\end{array}
;\frac{1}{\lambda^6} \right)_{\mathbb{F}_q} &(j\neq 0).
\end{cases}
\end{split}
\end{equation*}
\normalsize
\end{cor}
\begin{rem}
For the Dwork hypersurfaces of degree six, condition (\ref{assumption1}) in Theorem \ref{mmiyatani} is equivalent to $q \equiv 1 \bmod 6$, and conditions (\ref{assumption2}) and (\ref{assumption3}) are automatically satisfied. (See \cite[Example 3.3]{miyatani}.)
\end{rem}

From the definition of $\mathrm{Ker}(f_{q-1})$, an action of the symmetric group of degree six $S_6$ on $\mathrm{Ker}(f_{q-1})$ is naturally defined.
We put 
\footnotesize
$$\langle{}^t(s_1,\ldots,s_6)\rangle^k:=\left\{{}^t(v_1,\ldots,v_6)\in \{s_0,\ldots,s_{6^4-1}\}\mid
\begin{array}{l}
\text{There exists a permutation}\ \sigma \in S_6\\
\text{such that}\  {}^t(v_{\sigma(1)},\ldots,v_{\sigma(6)})={}^t(s_1,\ldots,s_6) 
\end{array}
\right\},$$\normalsize 
where 
\footnotesize
\begin{equation*}
k=\#\left\{ {}^t(v_1,\ldots,v_6)\in \{s_0,\ldots,s_{6^4-1}\}\mid
\begin{array}{l}
\text{There exists a permutation}\ \sigma \in S_6\\
\text{such that}\ {}^t(v_{\sigma(1)},\ldots,v_{\sigma(6)})={}^t(s_1,\ldots,s_6) 
\end{array}
\right\}.
\end{equation*}
\normalsize
Then, we have
\small
\begin{equation*}
\begin{split}
\{s_0\ldots,s_{6^4-1}\}
&=\langle{}^t(0,0,0,0,0,0)\rangle^1\cup\langle{}^t(0,0,0,0,t,5t)\rangle^{30}\cup\langle{}^t(0,0,0,0,2t,4t)\rangle^{30} \\
&\cup\langle{}^t(0,0,0,0,3t,3t)\rangle^{15}\cup\langle{}^t(0,0,0,t,t,4t)\rangle^{60}\cup\langle{}^t(0,0,0,t,2t,3t)\rangle^{120} \\
&\cup\langle{}^t(0,0,0,2t,2t,2t)\rangle^{20}\cup\langle{}^t(0,0,0,2t,5t,5t)\rangle^{60}\cup\langle{}^t(0,0,0,3t,4t,5t)\rangle^{120} \\
&\cup\langle{}^t(0,0,t,t,2t,2t)\rangle^{90}\cup\langle{}^t(0,0,2t,2t,4t,4t)\rangle^{30}\cup\langle{}^t(0,0,t,3t,4t,4t)\rangle^{180} \\
&\cup\langle{}^t(0,0,t,3t,3t,5t)\rangle^{180}\cup\langle{}^t(0,0,t,2t,4t,5t)\rangle^{360}.
\end{split}
\end{equation*}
\normalsize
Next, we calculate $A(s):=\gamma(s)F(s)$ for 
$$s={}^t(0,0,0,0,0,0),{}^t(0,0,0,0,t,5t),\ldots, {}^t(0,0,t,2t,4t,5t)$$
by using Theorem \ref{miyatani}.
\begin{prop}\label{miyatanimaccarthy}
We have the following identities:
\small
\begin{equation}
\begin{split}
A({}^t(0,0,0,0,0,0))=-{}_5\widetilde{F}_5\left(
\begin{array}{cccccc}
\omega^t& \omega^{2t} & \omega^{3t} & \omega^{4t} & \omega^{5t}\\
\epsilon& \epsilon& \epsilon & \epsilon &\epsilon \\
\end{array}
;\frac{1}{\lambda^6} \right)_{\mathbb{F}_q}
\end{split}
\end{equation}

\begin{equation}
\begin{split}
A({}^t(0,0,0,0,t,5t))=-q\omega^t(-1)\cdot{}_3\widetilde{F}_3\left(
\begin{array}{cccc}
\omega^{2t} & \omega^{3t} & \omega^{4t} \\
\epsilon & \epsilon &\epsilon \\
\end{array}
;\frac{1}{\lambda^{6}} \right)_{\mathbb{F}_q}
\end{split}
\end{equation}

\begin{equation}
\begin{split}
A({}^t(0,0,0,0,2t,4t))=-q\cdot{}_3\widetilde{F}_3\left(
\begin{array}{cccc}
\omega^{t} & \omega^{3t} & \omega^{5t} \\
\epsilon & \epsilon &\epsilon \\
\end{array}
;\frac{1}{\lambda^{6}} \right)_{\mathbb{F}_q}
\end{split}
\end{equation}

\begin{equation}
\begin{split}
A({}^t(0,0,0,0,3t,3t))=q\omega^t(-1)\cdot{}_4\widetilde{F}_4\left(
\begin{array}{ccccc}
\omega^{t} & \omega^{2t} & \omega^{4t} & \omega^{5t} \\
\epsilon & \epsilon &\epsilon & \omega^{3t} \\
\end{array}
;\frac{1}{\lambda^{6}} \right)_{\mathbb{F}_q}
\end{split}
\end{equation}

\begin{equation}
\begin{split}
A({}^t(0,0,0,t,t,4t))=-qJ(\omega^{2t},\omega^{5t},\omega^{5t})\cdot{}_3\widetilde{F}_3\left(
\begin{array}{cccc}
\omega^{2t} & \omega^{3t} & \omega^{5t} \\
\epsilon & \epsilon &\omega^t \\
\end{array}
;\frac{1}{\lambda^{6}} \right)_{\mathbb{F}_q}
\end{split}
\end{equation}

\begin{equation}
\begin{split}
A({}^t(0,0,0,2t,5t,5t))=-qJ(\omega^{t},\omega^{t},\omega^{4t})\cdot{}_3\widetilde{F}_3\left(
\begin{array}{cccc}
\omega^{3t} & \omega^{4t} & \omega^{t} \\
\epsilon & \epsilon &\omega^{5t} \\
\end{array}
;\frac{1}{\lambda^{6}} \right)_{\mathbb{F}_q}
\end{split}
\end{equation}

\begin{equation}
\begin{split}
A({}^t(0,0,0,2t,2t,2t))=-qJ(\omega^{4t},\omega^{4t},\omega^{4t})\cdot{}_4\widetilde{F}_4\left(
\begin{array}{ccccc}
\omega^{t} & \omega^{3t} & \omega^{4t} &\omega^{5t} \\
\epsilon & \epsilon &\omega^{2t} &\omega^{2t} \\
\end{array}
;\frac{1}{\lambda^{6}} \right)_{\mathbb{F}_q}
\end{split}
\end{equation}

\begin{equation}
\begin{split}
A({}^t(0,0,0,3t,4t,5t))=-qJ(\omega^{t},\omega^{2t},\omega^{3t})\cdot{}_2\widetilde{F}_2\left(
\begin{array}{ccc}
\omega^{2t} & \omega^{t}  \\
\epsilon & \epsilon  \\
\end{array}
;\frac{1}{\lambda^{6}} \right)_{\mathbb{F}_q}
\end{split}
\end{equation}

\begin{equation}
\begin{split}
A({}^t(0,0,0,t,2t,3t))=-qJ(\omega^{3t},\omega^{4t},\omega^{5t})\cdot{}_2\widetilde{F}_2\left(
\begin{array}{ccc}
\omega^{4t} & \omega^{5t}  \\
\epsilon & \epsilon  \\
\end{array}
;\frac{1}{\lambda^{6}} \right)_{\mathbb{F}_q}
\end{split}
\end{equation}

\begin{equation}
\begin{split}
A({}^t(0,0,t,t,2t,2t))=-qJ(\omega^{4t},\omega^{4t},\omega^{5t},\omega^{5t})\cdot{}_3\widetilde{F}_3\left(
\begin{array}{cccc}
\omega^{3t} & \omega^{4t} & \omega^{5t} \\
\epsilon & \omega^t &\omega^{2t} \\
\end{array}
;\frac{1}{\lambda^{6}} \right)_{\mathbb{F}_q}
\end{split}
\end{equation}

\begin{equation}
\begin{split}
A({}^t(0,0,2t,2t,4t,4t))=-q^2\cdot{}_3\widetilde{F}_3\left(
\begin{array}{cccc}
\omega^{3t} & \omega^{5t} & \omega^{t} \\
\epsilon & \omega^{2t} &\omega^{4t} \\
\end{array}
;\frac{1}{\lambda^{6}} \right)_{\mathbb{F}_q}
\end{split}
\end{equation}

\begin{equation}
\begin{split}
A({}^t(0,0,t,3t,4t,4t))=qJ(\omega^{2t},\omega^{2t},\omega^{3t},\omega^{5t})\cdot{}_2\widetilde{F}_2\left(
\begin{array}{ccc}
\omega^{2t} & \omega^{5t}  \\
\epsilon & \omega^{4t}  \\
\end{array}
;\frac{1}{\lambda^{6}} \right)_{\mathbb{F}_q}
\end{split}
\end{equation}

\begin{equation}
\begin{split}
A({}^t(0,0,t,3t,3t,5t))=-q^2\cdot{}_2\widetilde{F}_2\left(
\begin{array}{ccc}
\omega^{2t} & \omega^{4t}  \\
\epsilon & \omega^{3t}  \\
\end{array}
;\frac{1}{\lambda^{6}} \right)_{\mathbb{F}_q}
\end{split}
\end{equation}

\begin{equation}
\begin{split}
A({}^t(0,0,t,2t,4t,5t))=-q^2\omega^t(-1)\cdot{}_1\widetilde{F}_1\left(
\begin{array}{cc}
\omega^{3t}  \\
\epsilon   \\
\end{array}
;\frac{1}{\lambda^{6}} \right)_{\mathbb{F}_q} \label{1F1}
\end{split}
\end{equation}
\normalsize
\end{prop}

Finally, we rewrite Proposition \ref{miyatanimaccarthy} by using Greene's hypergeometric function. McCarthy gave the relation between his hypergeometric function and Greene's hypergeometric function.

\begin{prop}\cite[Proposition 2.5]{transformation}\label{transformation}
For characters $A_0$, $\ldots$, $A_n$, $B_1$, $\ldots$, $B_n$ with $A_0\neq \epsilon$ and $A_i\neq B_i$ $($$i=1,\ldots,n$$)$, we have
\footnotesize
\begin{equation*}
{}_{n+1}\widetilde{F}_{n+1}
\left(
\begin{array}{ccccc}A_0 &A_1 & \cdots & A_{n} \\
\epsilon & B_1& \cdots& B_{n}  \\
\end{array}
;x \right)_{\mathbb{F}_q}=\left(\prod_{i=1}^{n}\binom{A_i}{B_i}^{-1}\right){}_{n+1}F_n \left(
\begin{array}{cccc|c}
A_0, & A_1, & \ldots & A_n \\
& B_1, & \ldots & B_n
\end{array}
\ x \right)
_q.
\end{equation*}
\end{prop}

%\subsection{The proof of Theorem \ref{main theorem}}
Theorem \ref{main theorem} follows from using Propositions \ref{miyatanimaccarthy} and \ref{transformation} and applying the identity 
\begin{equation*}
{}_1F_0\left(
\begin{array}{cc}
\omega_{\alpha}  \\
\epsilon   \\
\end{array}
;x \right)_q=\varepsilon(x)\overline{\omega}_{\alpha}(1-x)
\end{equation*}
(See \cite[(3.11)]{Greene}) for the identity (\ref{1F1}).

\begin{rem}
We can also prove {\cite[Theorem 1.1]{case4}} and {\cite[Theorem 1.2]{cased}} similarly.
\end{rem}

\section*{Acknowledgements}
The author is indebted to Professor Shinichi Kobayashi, his supervisor, for his excellent guidance, patience and constant encouragement. He is grateful to Professor Kazuaki Miyatani  for informing the author about his result and giving essential advice. He would also like to Professor Noriyuki Otsubo, Akio Nakagawa and Hiroki Obama for valuable comments and discussion.


\begin{thebibliography}{99}
\bibitem{case4} H. Goodson, Hypergeometric functions and relations to Dwork hypersurfaces. Int.\ J. Number Theory 13 (2017), no. 2, 439–485.
\bibitem{cased} H. Goodson, A complete hypergeometric point count formula for Dwork hypersurfaces. J. Number Theory 179 (2017), 142–171.
\bibitem{McCarthy} D. McCarthy, On a supercongruence conjecture of Rodriguez-Villegas. Proc.\ Amer.\ Math.\ Soc.\ 140 (2012), no. 7, 2241–2254.
\bibitem{transformation} D. McCarthy, Transformations of well-poised hypergeometric functions over finite fields. Finite Fields Appl. 18 (2012), no. 6, 1133–1147.
\bibitem{McCarthy F_p point} D. McCarthy, The number of $\mathbb{F}_p$-points on Dwork hypersurfaces and hypergeometric functions. Res. Math. Sci. 4 (2017), Paper No. 4, 15 pp.
\bibitem{Koblitz} N. Koblitz, The number of points on certain families of hypersurfaces over finite fields. Compositio Math. 48 (1983), no. 1, 3–23.
\bibitem{Ireland} K. Ireland and M. Rosen, A classical introduction to modern number theory. Second edition. Graduate Texts in Mathematics, 84. Springer-Verlag, New York, 1990. 
\bibitem{Greene} J. Greene, Hypergeometric functions over finite fields. Trans.\ Amer.\ Math.\ Soc.\ 301 (1987), no. 1, 77–101.
\bibitem{MTY} K. Matsumoto, T. Terasoma and S. Yamazaki, Jacobi's formula for Hesse cubic curves. Acta Math. Vietnam. 35 (2010), no. 1, 91–105.
\bibitem{handbook} G. L. Mullen and D. Panario, Handbook of finite fields,  Discrete Mathematics and its Applications (Boca Raton). CRC Press, Boca Raton, FL, 2013.
\bibitem{lang} S. Lang, Cyclotomic Fields I and II. Graduate Texts in Mathematics, vol. 121, Springer-Verlag, New York, 1990. 
\bibitem{Weil} A. Weil, Numbers of solutions of equations in finite fields. Bull.\ Amer.\ Math.\ Soc. 55
(1949), 497–508.
\bibitem{miyatani} K. Miyatani, Monomial deformations of certain hypersurfaces and two hypergeometric functions. Int.\ J. Number Theory 11 (2015), no. 8, 2405–2430.
\bibitem{katz} N. M. Katz, Exponential sums and differential equations. Annals of Mathematics Studies, 124. Princeton University Press, Princeton, NJ, 1990.
\bibitem{Salerno} A. Salerno, Counting points over finite fields and hypergeometric functions. Funct. Approx. Comment. Math. 49 (2013), no. 1, 137–157.
\end{thebibliography}
\end{document}